\documentclass[12pt, a4paper]{amsart}

\usepackage{amsmath}
\usepackage{geometry,amsthm,graphics,tabularx,amssymb,shapepar}
\usepackage{amscd}
\usepackage{graphicx}
\usepackage[all,2cell,dvips]{xy}

\newcommand{\CX}{{\mathcal {X}}}
\newcommand{\CY}{{\mathcal {Y}}}
\newcommand{\CZ}{{\mathcal {Z}}}

\newcommand{\rk}{{\mathrm{k}}}

\newcommand{\SL}{{\mathrm{SL}}}

\newcommand{\vsp}{{\vspace{0.2in}}}

\newcommand{\con}{\textit{C}}

\newcommand{\oL}{\operatorname{L}}

\newcommand{\oD}{\textit{D}}

\newcommand{\Z}{\mathbb{Z}}
\newcommand{\C}{\mathbb{C}}
\newcommand{\R}{\mathbb R}

\newcommand{\la}{\langle}
\newcommand{\ra}{\rangle}

\newcommand{\be}{\begin {equation}}
\newcommand{\ee}{\end {equation}}
\newcommand{\bee}{\begin {equation*}}
\newcommand{\eee}{\end {equation*}}

\newcommand{\supp}{\operatorname{supp}}

\theoremstyle{Theorem}
\newtheorem{introtheorem}{Theorem}

\theoremstyle{plain}
\newtheorem{thm}{Theorem}[section]

\newtheorem{lem}[thm]{Lemma}
\newtheorem{prop}[thm]{Proposition}

\title[Fourier transform and distributions]{Fourier transform and rigidity of certain distributions}

\author{Binyong Sun}
\address{Academy of Mathematics and Systems Science\\
Chinese Academy of Sciences\\
Beijing, 100190, P.R. China} \email{sun@math.ac.cn}

\author{Chen-Bo Zhu}
\address{Department of Mathematics\\
National University of Singapore\\
Block S17, 10 Lower Kent Ridge Road,
Singapore 119076} \email{matzhucb@nus.edu.sg}
\begin{document}


\begin{abstract} Let $E$ be a finite dimensional vector space over a local field,
and $F$ be its dual. For a closed subset $X$ of $E$, and $Y$ of $F$,
consider the space $\textit{D}^{-\xi}(E;X,Y)$ of tempered
distributions on $E$ whose support are contained in $X$ and support
of whose Fourier transform are contained in $Y$. We show that
$\textit{D}^{-\xi}(E;X,Y)$ possesses a certain rigidity property,
for $X$, $Y$ which are some finite unions of affine subspaces.
\end{abstract}

\maketitle

\section{Introduction and main result}

One of the most fundamental results in Euclidean harmonic analysis
is the uncertainty principle. As a meta-theorem, it states that a
nonzero function and its Fourier transform cannot both be sharply
localized. There are various concrete formalisations of this
principle, most famously the Heisenberg uncertainty principle. For a
lively and insightful discussion of this topic, see \cite{Tao}.

We shall consider distributions (\cite{Sch}). Fix a finite
dimensional vector space $E$ and its dual $F$. For a closed subset
$X$ of $E$, and $Y$ of $F$, consider the space $\oD^{-\xi}(E;X,Y)$
of tempered distributions on $E$ whose support are contained in $X$
and support of whose Fourier transform are contained in $Y$. The
general thrust of the current note is to examine to what extent one
is able to separate the support and support of the Fourier transform
for distributions in $\oD^{-\xi}(E;\cup \CX,\cup \CY)$, where $\CX $
and $\CY $ are finite sets of affine subspaces. This is the meaning
of rigidity in the title, which the authors consider as another
instance of the uncertainty principle.

One key observation, which surely has been noted by others before
us, is the general importance of relative position of the pair
$(X,Y)$, now assumed to be affine subspaces. As it turns out, there
will be three different circumstances, which we respectively call
thin, perfect, thick (Equation \eqref{def-tpt}). As an indication of
the relevance of these concepts, we have: (a) $\oD^{-\xi}(E;X,Y)=0$
if $(X,Y)$ is a thin pair; (b) $\oD^{-\xi}(E;X,Y)$ can be explicitly
described in terms of a countable linear basis if $(X,Y)$ is a
perfect pair. This is similar to the classical result of L. Schwartz
on the structure of distributions supported on a single point; (c)
If $(X,Y)$ is a thick pair, then $\oD^{-\xi}(E;X,Y)$ contains (in a
non-canonical fashion) the space of tempered distributions on a
nonzero subspace of $E$ and is thus not rigid in any reasonable
sense.

Our main result (Theorem A in this section) is that
$\oD^{-\xi}(E;\cup \CX,\cup \CY)$ possesses the afore-mentioned
rigidity property as long as there is no pair $(X,Y)\in \CX\times
\CY$ which is thick. Very roughly the idea goes as follows: By
applying a good multiplier operator (a suitable function which
vanishes on a part of the support of the Fourier transform), one may
cut off that part of the support in the Fourier transform side. At
the distribution side, the process will generally yield a
distribution with additional support. If no thick pairs are
involved, then this process can be carried out in such a way that
the additional support is very much controlled. We will have more to
say on the detailed strategy later.

The result of this note was motivated by certain
representation-theoretic issues arising from the proof of
archimedean multiplicity-one theorems \cite{SZ}. More specifically
the rigidity statement of Theorem A allows us to establish the
semi-simplicity and non-negativity of an Euler vector field, crucial
in certain descent step involving the so-called distinguished
nilpotent orbits. For applications to representation theory of
algebraic groups over non-archimedean local fields, we will prove
our results over an arbitrary local field, rather than over $\R$ or
$\C$. We note that in Euclidean harmonic analysis and integral geometry,
it is of substantial interest to calculate the Fourier (or Radon) transform of
distributions with support in an algebraic set. We refer the interested reader to the classical book
``Generalized functions" by Gel'fand and Shilov \cite[vol. 1 and 4]{GS}.
We would also like to point out that for a real quadratic space $E$ and $X=Y$ the null cone,
the space $\oD^{-\xi}(E;X,Y)$ is of special interest for a
variety of reasons. Some related topics are explored in \cite[Chapter 4]{HT}.

\vsp We now introduce some necessary notation for this note.

Let $\rk$ be an arbitrary local field, and let $\psi:\rk\rightarrow \C^\times$ be a fixed nontrivial unitary
character. Let $E$ be a finite dimensional $\rk$-vector space.
Denote by
\[
   \con^{\,\varsigma}(E)\subset \con^{-\xi}(E), \quad \oD^{\,\varsigma}(E)\subset \oD^{-\xi}(E)
\]
the (complex) spaces of Schwartz functions, tempered generalized functions, Schwartz densities, and tempered distributions on $E$, respectively.
Thus $\con^{-\xi}(E)$ (resp., $\oD^{-\xi}(E)$) is the dual of $\oD^{\,\varsigma}(E)$ (resp., $\con^{\,\varsigma}(E)$).

Let $F$ be a finite-dimensional $\rk$-vector space which is dual to
$E$, i.e., a non-degenerate bilinear map
\[
  \la \,,\,\ra: E\times F\rightarrow \rk
\]
is given. The Fourier transform
\[
  \begin{array}{rcl}
    \oD^{\,\varsigma}(F)&\rightarrow &\con^{\,\varsigma}(E)\\
        \omega &\mapsto &\hat{\omega}
    \end{array}
\]
is the linear isomorphism given by
\[
   \hat{\omega}(x):=\int_{F} \psi(\la x,y\ra)\,
   \omega(y),\quad x\in E.
\]
Dually for every $D\in \oD^{-\xi}(E)$, its Fourier transform $\widehat
D\in \con^{-\xi}(F)$ is given by
\[
  \widehat{D}(\omega):=D(\hat{\omega}), \quad \omega\in \oD^{\,\varsigma}(F).
\]

For every closed subset $X$ of $E$, and $Y$ of $F$, denote
\[
   \oD^{-\xi}(E;X):=\{D\in \oD^{-\xi}(E)\mid \supp(D)\subset X\},
\]
and
\begin{equation}
\label{def-XY}
   \oD^{-\xi}(E;X,Y):=\{D\in \oD^{-\xi}(E;X)\mid \supp(\widehat D)\subset Y\}.
\end{equation}

Let $X$ be an affine subspace of $E$, and $Y$ an affine subspace of
$F$. Denote
\[
  \oL(X):=\{u-v\mid u,v\in X\}
\]
the subspace associated to $X$, and likewise $\oL(Y)$ for $Y$. We
say that the pair $(X,Y)$ is thick, perfect or thin according as
\begin{equation}
\label{def-tpt}
  \oL(X)^\perp\subsetneq \oL(Y),\quad \oL(X)^\perp=\oL(Y),\quad \textrm{or
  } \oL(X)^\perp\nsubseteq \oL(Y),
\end{equation}
or what is the same,
\[
  \oL(Y)^\perp\subsetneq \oL(X),\quad \oL(Y)^\perp=\oL(X),\quad \textrm{or
  } \oL(Y)^\perp\nsubseteq \oL(X),
\]
respectively.

\vsp

\noindent {\bf Example}: Take $F$ to be a non-degenerate quadratic
space. Suppose that $F_0$ is a non-degenerate nonzero subspace of
$F$ and
\[
  (F_0)^\perp=F^+\oplus F^-
\]
is a decomposition into totally isotropic subspaces $F^+$ and $F^-$. Then the pairs $(F^+, F^+)$, $(F^+\oplus F_0, F^+)$, and $(F^+\oplus F_0, F^+\oplus F_0)$ are thin, perfect and thick, respectively.

\vsp

\noindent {\bf Remark}: If $(X,Y)$ is a perfect pair, then
\[
   \oD^{-\xi}(E;X,Y)=\left\{
                           \begin{array}{ll}
                            \C\, \mu_{X,y_0},\quad&\textrm{if $\rk$ is nonarchimedean},\smallskip\\
                             (\C[X] \,\mu_{X,y_0}) \otimes \oD^{-\xi}(L';\{0\}),\quad&\textrm{if $\rk$ is archimedean.}
                           \end{array}
                       \right.
\]
Here $y_0\in Y$ is an arbitrary element, $\mu_{X,y_0}\in \oD^{-\xi}(X)$ is the product of the function $\psi(\la \cdot, -y_0\ra)$ with a fixed $\oL(X)$-invariant positive Borel measure on $X$. In the archimedean case, $L'$ is an arbitrary subspace of $E$ such that $\oL(X)\oplus L'=E$. Through addition we have a decomposition $E=X\times L'$. The space $\C[X]$ is then the algebra of (complex valued) polynomial functions on $X$, viewed as a real affine space.

\vsp
We now state the main result of this note.

\begin{introtheorem}\label{introtheorem}
Let $\CX$ be a finite set of affine subspaces of $E$, and $\CY$ a finite set of affine subspaces of $F$. Assume that
there is no pair $(X,Y)\in \CX\times \CY$ which is thick. Then
\[
   \oD^{-\xi}(E;\cup \CX,\cup \CY)=\bigoplus_{(X,Y)\in \CX\times \CY\,\, \textrm{that is perfect}} \oD^{-\xi}(E;X,Y).
\]

\end{introtheorem}

\vsp

\noindent {\bf Remarks}: (i) Theorem A asserts in particular that
$\oD^{-\xi}(E;\cup \CX,\cup \CY)=0$, if every pair $(X,Y)\in
\CX\times \CY$ is thin. (ii) In the archimedean case,
$\oD^{-\xi}(E;X,Y)$ is a module for the Weyl algebra of $E$
(consisting of (complex) polynomial coefficient differential
operators on $E$). We note that for a perfect pair $(X,Y)$, the Weyl
algebra module $\oD^{-\xi}(E;X,Y)$ is irreducible.

\vsp

The strategy to prove Theorem A is to control support and it goes as follows.
For every vector $u\in E$, define the following function on $F$:
\begin{equation}
\label{defphi}
\phi _u:=\psi(\la u,\,\cdot\ra).
\end{equation}

Take a distribution $D\in \oD^{-\xi}(E;\cup \CX,\cup \CY)$. For any
$Y\in \CY$, pick a nonzero $u_Y\in \oL(Y)^\perp$. The function
$\phi_{u_Y}$ takes a constant value on $Y$, which we denote by
$c_Y$. Thus $\phi_{u_Y}-c_Y$ will vanish on $Y$, and multiplying a
high power of $\phi_{u_Y}-c_Y$ will cut $Y$ out of the support of
the Fourier transform of $D$. The result is the Fourier transform of
a new distribution which is a linear combination of $D$ and
translates of $D$ by multiples of $u_Y$. Doing this consecutively
for different $Y$'s in $\CY$ will thus yield a distribution which is
a linear combination of $D$ and translates of $D$ by elements of the
lattice in $\oL(Y)^\perp$ generated by $u_Y$'s, and which has a
significantly reduced support for its Fourier transform. If $X \in
\CX$ is thin with respect to some $Y$'s, then one could arrange the
$u_Y$'s (c.f. Lemma \ref{fam}) so that the lattice generated by
$u_Y$'s is in a favorable position relative to $X$, resulting in an
excellent control on the support of the new distribution.

\vsp Here are some words on the organization of this note. In
Section \ref{thin-pairs}, we show that if $X\in \CX$ has the
property that $(X,Y)$ is thin for every $Y\in \CY$, then $X$ in fact
does not appear in the support of any $D\in \oD^{-\xi}(E;\cup
\CX,\cup \CY)$. This is a form of the uncertainty principle. In
Section \ref{pure-affine-pairs}, we show that the rigidity property
as claimed in Theorem A holds in the special case, when $\CX$ (resp. $\CY$) is a finite set of translations of a
subspace $X_0$ of $E$ (resp. a subspace $Y_0$ of $F$). Section
\ref{general-affine-pairs} is devoted to the general case. We prove an inductive step in Proposition \ref{dec2}.
Theorem \ref{introtheorem} will then follow quickly from the results of Sections \ref{thin-pairs} and \ref{pure-affine-pairs}, and the induction result just alluded to.

\vsp 
\noindent {\it Acknowledgements.} The authors thank the referee for helpful comments.  Binyong Sun is supported by NSFC grants 10801126 and 10931006. Chen-Bo Zhu
is supported by MOE2010-T2-2-113.

\section{Thin pairs and elimination of support}
\label{thin-pairs}

Let $\CX$ be a finite set of affine subspaces of $E$, and let
$\CY$ be a finite set of affine subspaces of $F$, as in the Introduction. Assume that both
$\CX$ and $\CY$ are nonempty.

Note that an integer may also be considered as an element of $\rk$. For any family $\mathbf a=\{a_Y\}_{Y\in \CY} \in \Z ^\CY$, denote $[\mathbf a]$ the corresponding element of $\rk^\CY$.

\begin{lem}\label{fam} Assume that $X_1\in \CX$ and that $(X_1,Y)$ is thin for every $Y\in
\CY$. Let $x_1\in X_1\setminus \cup(\CX\setminus\{X_1\})$. Then
there is a family $\{u_Y\in \oL(Y)^\perp\}_{Y\in \CY}$ of vectors in
$E$ with the following property: for every $\mathbf a=\{a_Y\}_{Y\in
\CY}\in \Z ^\CY$ with $[\mathbf a]\neq 0$, we have
\[
 \sum_{Y\in \CY} a_Y u_Y\notin \cup \CX- x_1.
\]
\end{lem}
\begin{proof}  For every $X\in \CX$ and every $\mathbf a=\{a_Y\}_{Y\in
\CY}\in \Z ^\CY$ with $[\mathbf a]\neq 0$,
put
\[
  S_{X,\mathbf a}:=\left\{\{u_Y\}\in
\prod_{Y\in \CY} \oL(Y)^\perp\mid \sum_{Y\in \CY} a_Y u_Y\in
X-x_1\right\}.
\]
If $X\neq X_1$, then $0\notin S_{X,\mathbf a}$, and $S_{X,\mathbf
a}$ is a proper affine subspace of $\prod_{Y\in \CY} \oL(Y)^\perp$.
If $X=X_1$, then $S_{X,\mathbf a}$ is a subspace of
$\prod_{Y\in \CY} \oL(Y)^\perp$, and is proper due to the hypothesis that $(X_1, Y)$ is thin for every $Y\in
\CY$. In any case each $S_{X,\mathbf a}$ is a
measure zero set of $\prod_{Y\in \CY} \oL(Y)^\perp$, and so is the (countable) union $\cup_{X\in \CX,\,[\mathbf a]\neq
  0} S_{X,\mathbf a}$. We finish the
proof by taking a vector in
\[
  (\prod_{Y\in \CY} \oL(Y)^\perp)\setminus (\cup_{X\in \CX,\,[\mathbf a]\neq
  0} S_{X,\mathbf a}).
\]
\end{proof}

For every vector $u\in E$, denote by
$T_u:\oD^{-\xi}(E)\rightarrow\oD^{-\xi}(E)$  the push forward of the translation by
$u$, and write
\[
  (T_u f)(\omega):=f(T_u \omega), \quad f\in \con^{-\xi}(E), \,\omega\in \oD^{\,\varsigma}(E).
\]
Similar notation applies for $v\in F$.

The following is a form of the uncertainty principle.
\begin{prop}\label{vectors}
Assume that $X_1\in \CX$ and that $(X_1,Y)$ is thin for every $Y\in
\CY$. Then we have
\[
  \oD^{-\xi}(E;\cup \CX,\cup \CY)=\oD^{-\xi}(E;\cup (\CX\setminus \{X_1\}),\cup
 \CY).
\]
Consequently we have
\[
 \oD^{-\xi}(E;\cup \CX,\cup \CY)=\oD^{-\xi}(E;\cup \CX',\cup
 \CY'), \]
where
\[
\begin{aligned}
  \CX':=\{X\in \CX\mid (X,Y)\textrm{ is not thin for some }Y\in \CY\}, \\
  \CY':=\{Y\in \CY\mid (X,Y) \textrm{ is not thin for some }X\in \CX\}.
\end{aligned}
\]
\end{prop}

\begin{proof}
Let $x_1\in X_1\setminus
\cup(\CX\setminus\{X_1\})$ and $\{u_Y\in \oL(Y)^\perp\}_{Y\in \CY}$ be as in Lemma
\ref{fam}. Let $\phi_{u_Y}$ be as in \eqref{defphi}, and $c_Y$ be its common value on $Y$, as in the Introduction.

Let $D\in \oD^{-\xi}(E;\cup \CX,\cup \CY)$. Then
\[
  \prod_{Y\in \CY} (\phi_{u_Y}-c_Y)^m \widehat D=0,
\]
where $m=1$ if $\rk$ is nonarchimedean, and $m$ is a sufficiently
large positive integer if $\rk$ is archimedean. The above equality
is equivalent to
\[
  \prod_{Y\in \CY} (T_{u_Y}-c_Y)^m D=0,
\]
or what is the same,
\begin{equation}\label{d0}
\sum_{\mathbf a=\{a_Y\}\in \{0,1,\cdots,m\}^{\CY}} c_{\mathbf a}
T_{u_{\mathbf a}}
  D=0,
\end{equation}
where
\[
  c_{\mathbf a}:=\prod_{Y\in \CY} {{m}\choose{a_Y}}
  (-c_Y)^{m-a_Y},
\]
and
\[
u_{\mathbf a}:=\sum_{Y\in \CY} a_Y u_Y.
\]

The choice of $\{u_Y\}$ ensures that $-u_{\mathbf a}\notin \cup
\CX-x_1$ whenever $[\mathbf a]$ is nonzero. Let $U$ be an open
neighborhood of $0$ in $E$, small enough so that
\[
  (-u_{\mathbf a}+U) \cap
(\cup \CX-x_1)=\emptyset, \quad\textrm{for all $\mathbf a\in \{0,1,\cdots,m\}^{\CY}$ such that $[\mathbf a]\ne 0$.}
\]
Since $D$ is supported in $\cup \CX$, this
implies that
\[
   (T_{u_{\mathbf a}}D)|_{x_1+U}=0, \quad\textrm{for all $\mathbf a\in \{0,1,\cdots,m\}^{\CY}$ such that $[\mathbf a]\ne 0$.}
\]
Together with \eqref{d0}, this implies that $D|_{x_1+U}=0$. Since
$x_1$ is arbitrary, we conclude that $D$ is supported in
$\cup(\CX\setminus \{X_1\})$.
\end{proof}

\section{A special case}
\label{pure-affine-pairs}

\begin{prop}\label{trans}
If $\CX$ is a finite set of translations of a subspace $X_0$ of
$E$, and $\CY$ is a finite set of translations of a subspace $Y_0$
of $F$, then
\begin{equation}\label{d1}
   \oD^{-\xi}(E;\cup \CX,\cup \CY)=\bigoplus_{(X,Y)\in \CX\times \CY} \oD^{-\xi}(E;X,Y).
\end{equation}
\end{prop}
\begin{proof}
Assume that $X_0^\perp \subseteq Y_0$, i.e., $(X_0,Y_0)$ is not a thin pair. Otherwise both sides of
\eqref{d1} are $0$ by Proposition \ref{vectors}, and there is nothing to prove.

Let $D\in  \oD^{-\xi}(E;\cup \CX, \cup \CY)$. For every $X\in \CX$, denote by
$D_X \in \oD^{-\xi}(E;X)$ the distribution which coincides with $D$ on a neighborhood of $X$. Then
\[
  D=\sum_{X\in\CX} D_X.
\]
Note that $\cup \CX$ is a disjoint union, by our assumption on $\CX$.

Assume that $\rk$ is archimedean. Let $P$ be a real polynomial function whose zero locus is $\cup \CY$.  Then  there is a positive integer $k$ such that $P^k \widehat D=0$, or equivalently, $(\widetilde P)^k D=0$, where $\widetilde P$ is a certain constant coefficient differential operator, acting on the space $\oD^{-\xi}(E)$. Therefore for all $X\in \CX$,
$(\widetilde P)^k D_X=0$, which implies that $D_X\in \oD^{-\xi}(E;X,\cup \CY)$. This proves that
\begin{equation}\label{d2}
   \oD^{-\xi}(E;\cup \CX,\cup \CY)=\bigoplus_{X\in \CX} \oD^{-\xi}(E;X,\cup \CY).
\end{equation}

Now assume that $\rk$ is nonarchimedean. Fix $X_1\in \CX$. For any
$X\in \CX\setminus \{X_1\}$, choose a vector $v_X\in X_0^\perp$
such that
\[
 \psi(\la X, v_X\ra)\neq\psi(\la X_1, v_X\ra).
\]
Here $\psi(\la X, v_X\ra)$ stands for $\psi(\la u, v_X\ra)$, which
is independent of $u\in X$, and $\psi(\la X_1, v_X\ra)$ is defined
similarly. Then we have that
\[
  \left( \prod_{X\in \CX\setminus \{X_1\}}(\phi_{v_X}-\psi(\la X,
   v_X\ra))\right)D=\left( \prod_{X\in \CX\setminus \{X_1\}}(\psi(\la X_1, v_X\ra-\psi(\la X,
   v_X\ra))\right)D_{X_1}.
\]
(This is not true when $\rk$ is archimedean.)

The Fourier transform of the left hand side of the above equality
is
\[
   \left( \prod_{X\in \CX\setminus \{X_1\}}(T_{v_X}-\psi(\la X,
   v_X\ra))\right)\widehat D.
\]
Since $v_X\in X_0^\perp \subseteq Y_0$, and since $\cup \CY$ is invariant under translation by elements of $Y_0$, the above generalized function is again supported in $\cup \CY$. Therefore the Fourier transform
of $D_{X_1}$ is also supported in $\cup \CY$. This proves
\eqref{d2} in the nonarchimedean case.

Applying \eqref{d2} to the pair $\CY$ and $\{X\}$, we have that
\begin{equation}\label{d3}
   \oD^{-\xi}(E;X,\cup \CY)=\bigoplus_{Y\in \CY} \oD^{-\xi}(E;X,Y).
\end{equation}
We finish the proof by combining \eqref{d2} and \eqref{d3}.

\end{proof}

\section{The general case}
\label{general-affine-pairs} As before, let $\CX$ be a finite set of
affine subspaces of $E$, and let $\CY$ be a finite set of affine
subspaces of $F$. We start with the following

\begin{prop}\label{dec2}

Let $X_0$ be a subspace of $E$ and $Y_0$ a subspace of
$F$. Write
\[
  \CX_0:=\{X\in \CX\mid \oL(X)=X_0\}\,\textrm{ and }\,\CY_0:=\{Y\in \CY\mid \oL(Y)=Y_0\}.
\]
Assume that $(X',Y_0)$ and $(X_0,Y')$ are thin for all
\[
 X'\in \CX':=\CX\setminus \CX_0\,\textrm{ and all }\,Y'\in \CY':=\CY\setminus
\CY_0. \]
Then
\begin{equation}\label{dxy}
   \oD^{-\xi}(E;\cup \CX, \cup \CY)=
   \oD^{-\xi}(E;\cup \CX_0, \cup \CY_0)\oplus \oD^{-\xi}(E;\cup \CX', \cup
   \CY').
\end{equation}
\end{prop}

\begin{proof}
Proposition \ref{vectors} implies that the right hand side of
\eqref{dxy} is a direct sum.

Let $\{u_{Y'}\in \oL(Y')^\perp\}_{Y'\in\CY'}$ be a family of vectors
in $E$ such that for every family $\mathbf a=\{a_{Y'}\}_{Y'\in \CY'}\in \Z^{\CY'}$ with $[\mathbf a]\ne 0$, we have
\[
 \sum_{Y'\in \CY'} a_{Y'} u_{Y'}\notin \bigcup_{X_1,X_2\in
 \CX_0}(X_1-X_2).
\]
The existence of such a family is proved along the same line as that of Lemma
\ref{fam}. Let $\phi_{u_{Y'}}$ be as in \eqref{defphi}, and $c_{Y'}$ be its common value on $Y'$, as in the Introduction.

Let $D\in \oD^{-\xi}(E;\cup \CX, \cup \CY)$. We take $m$ to be a
sufficiently larger positive integer if $\rk$ is archimedean and
$m=1$ if $\rk$ is non-archimedean. Then the generalized function
\[
  \left(\prod_{Y'\in \CY'} (\phi_{u_{Y'}}-c_{Y'})^m\right)\widehat D
\]
is supported in $\cup \CY_0$. It is the Fourier transform of the
distribution
\[
D':=\left(\prod_{Y'\in \CY'} (T_{u_{Y'}}-c_{Y'})^m\right) D.
\]
By expansion, we have
\begin{equation}\label{d4}
D'=\sum_{\mathbf a=\{a_{Y'}\}\in
\{0,1,\cdots,m\}^{\CY'}} c_{\mathbf a} T_{u_{\mathbf a}}
  D,
\end{equation}
where
\[
  c_{\mathbf a}:=\prod_{Y'\in \CY'} {{m}\choose{a_{Y'}}}
  (-c_{Y'})^{m-a_{Y'}},
\]
and
\[
  u_{\mathbf a}:=\sum_{Y'\in \CY'} a_{Y'} u_{Y'}.
\]

For every set $\CZ$ of affine subspaces of $E$, we put
\[
  \widetilde{\CZ}:=\left\{u_{\mathbf a}+Z\mid \mathbf a\in \{0,1,\cdots,m\}^{\CY'},\,Z\in
  \CZ\right\}.
\]
Then $D'$ is clearly supported in $\cup \widetilde \CX$. Its Fourier transform is supported in $\cup \CY_0$, and since
$(X',Y_0)$ is thin for all $X'\in \CX'$, Proposition \ref{vectors} implies that it is supported in $\cup
\widetilde{\CX}_0$.

Now the choice of $\{u_{Y'}\}$ ensures that
\[
  \widetilde{\CX_0}=\CX_0\sqcup \CX_1\quad \textrm{(a disjoint union)},
\]
where
\[
  \CX_1:=\left\{u_{\mathbf a}+X \mid \mathbf a\in \{0,1,\cdots,m\}^{\CY'},\,[\mathbf a]\neq 0,\,X\in
  \CX_0\right\}.
\]
By Proposition \ref{trans}, we may write
\begin{equation}\label{dec}
  D'=D_0+D_1,
\end{equation}
where $D_0\in \oD^{-\xi}(E;\cup \CX_0, \cup \CY_0)$ and $D_1\in
\oD^{-\xi}(E;\cup \CX_1, \cup \CY_0)$.

Let
\[
  x_0\in \cup \CX_0\setminus \cup \widetilde{\CX'}.
\]
The disjointness of $\CX_0$ and $\CX_1$ allows us to choose an open neighborhood $U$ of $0$ in $E$, small enough so
that
\[
  (x_0+U)\cap (\cup\widetilde{\CX'})=\emptyset,
\]
and
\begin{equation}\label{ints1}
  (x_0+U)\cap (\cup\CX_1)=\emptyset.
\end{equation}
For all nonzero $[\mathbf a]$, we thus have
\[
  (x_0+U-u_{\mathbf{a}})\cap (\cup\CX)= \emptyset
  \]
  and therefore
\[
  (T_{u_{\mathbf{a}}}D)|_{x_0+U}=0.
\]
Then \eqref{d4} implies that
\[
  D'|_{x_0+U}=c_{\mathbf 0} D|_{x_0+U}, \quad \textrm{with }c_{\mathbf 0}=\prod_{Y'\in \CY'}  (-c_{Y'})^m,
\]
and \eqref{dec} and \eqref{ints1} implies that
\[
  (D'-D_0)|_{x_0+U}=D_1|_{x_0+U}=0.
\]
We thus conclude from the last two equalities that $D-D_0/c_{\mathbf 0}$
vanishes on $x_0+U$. Since $x_0$ is arbitrary, we see that
 $D-D_0/c_{\mathbf 0}$ is supported in
\[
 (\cup \CX)\setminus (\cup \CX_0\setminus \cup \widetilde{\CX'})\subset\left(\cup \CX'\right)\cup(\cup(\CX_0\wedge \widetilde{\CX'}))
 .
\]
where
\[
 \CX_0\wedge \widetilde{\CX'}=\{X_0\cap X'\mid X_0\in \CX_0,\,X'\in
 \widetilde{\CX'}\}.
\]

Since every pair in $(\CX_0\wedge \widetilde{\CX'})\times \CY$
is thin, and every pair in $\CX'\times \CY_0$ is thin, we see that
 $D-D_0/c_{\mathbf 0}$ actually belongs to
\[
   \oD^{-\xi}(E;(\cup \CX')\cup(\cup(\CX_0\wedge \widetilde{\CX'})), \cup
  \CY)=\oD^{-\xi}(E;\cup \CX', \cup\CY)=\oD^{-\xi}(E;\cup \CX',
  \cup\CY'),
\]
by two applications of Proposition \ref{vectors}. This finishes the
proof of the current proposition.
\end{proof}

We are now ready to prove Theorem \ref{introtheorem}. Assume that
no pair in $\CX\times \CY$ is thick. Put
 \[
   \mathcal L:=\{ \oL(X)\mid (X,Y)\in \CX\times \CY \textrm{ is
   perfect}\}.
 \]
 For every $L\in \mathcal L$, put
 \[
  \CX_L:=\{X\in \CX\mid \oL(X)=L\}\quad \textrm{and}\quad \CY_{L^\perp}:=\{Y\in \CY\mid \oL(Y)=L^\perp\}.
 \]
Then we have
\begin{eqnarray*}
   && \oD^{-\xi}(E;\cup \CX,\cup \CY) \\
   &=&\oD^{-\xi}(E;\cup(\cup_{L\in \mathcal{L}} \CX_L),\cup(\cup_{L\in \mathcal{L}}
 \CY_{L^\perp}) )\qquad \textrm{(by Proposition \ref{vectors})}\\
 &=& \bigoplus_{L\in \mathcal{L}} \oD^{-\xi}(E; \cup \CX_L,\cup\CY_{L^\perp}) \qquad \qquad \qquad \textrm{(by Proposition \ref{dec2})}  \\
  &=& \bigoplus_{L\in \mathcal{L},\,(X,Y)\in  \CX_L\times \CY_{L^\perp}} \oD^{-\xi}(E; X, Y)) \qquad \textrm{(by Proposition \ref{trans})} \\
   &=& \bigoplus_{(X,Y)\in \CX\times \CY\,\, \textrm{that is perfect}}
\oD^{-\xi}(E;X,Y).
\end{eqnarray*}
This finishes the proof of  Theorem \ref{introtheorem}.

\end{document}